\documentclass[a4paper,12pt,leqno]{amsart}
\usepackage{amsmath,amsthm,amssymb,latexsym,epsfig,graphicx,subfigure}
\numberwithin{equation}{section}

\newtheorem{thm}{Theorem}[section]

\newtheorem{lm}[thm]{Lemma}

\theoremstyle{definition}

\theoremstyle{definition}

\newtheorem{ex}[thm]{Example}

\newcommand{\Rn}{\mathbb{R}^{n}}

\newcommand{\R}{\mathbb{R}}

\newcommand {\grtrsim} {\ {\raise-.5ex\hbox{$\buildrel>\over\sim$}}\ }
\newcommand{\e}{\varepsilon}

\newcommand{\khii}{\text{\lower -.4ex\hbox{$\chi$}}}
\DeclareMathOperator{\spt}{spt}

\renewcommand{\a}{\alpha}

\begin{document}
\title {Exceptional set estimates for the Hausdorff dimension of intersections}
\author{Pertti Mattila}

\thanks{The author was supported by the Academy of Finland.} 

\begin{abstract} 
Let $A$ and $B$ be Borel subsets of the Euclidean $n$-space with $\dim A + \dim B > n$ and let $0<u<\dim A + \dim B - n$ where $\dim$ denotes Hausdorff dimension. Then there is a set $E\subset O(n)$ of orthogonal transformations such that for $g\in O(n)\setminus E$,~   $\dim A\cap (g(B)+z) > u$ for $z$ in a set of positive Lebesgue measure. If $\dim A + \dim B > n+1$, then $\dim E\leq n(n-1)/2+1-u$, and if $\dim A \leq (n-1)/2$, then $\dim E\leq n(n-1)/2-u$. If $A$ is a Salem set and $0<u<\dim A + \dim B -n$ and $\dim A + \dim B > 2n-1$,  then $\dim A\cap (B+z) > u$ for $z$ in a set of positive Lebesgue measure. If $\dim A + \dim B \leq 2n-1$, the set of exceptional $g\in O(n)$ has dimension at most $n(n-1)/2-u$.

\end{abstract}

\maketitle

\section{Introduction} 

We let $O(n)$ denote the orthogonal group of the Euclidean $n$-space $\Rn$ and $\theta_n$ its Haar probability measure. We metrize $O(n)$ with the usual operator norm. Let also $\mathcal L^n$ stand for the Lebesgue measure on $\Rn$ and let $\dim$ stand for the Hausdorff dimension and $\mathcal H^s$ for $s$-dimensional Hausdorff measure. We shall prove the following theorem:

\begin{thm}\label{theo1}
Let $s$ and $t$ be positive numbers with $s+t > n+1$. Let $A$ and $B$ be Borel subsets of $\Rn$ with $\mathcal H^s(A)>0$ and $\mathcal H^t(B)>0$. Then 
there is $E\subset O(n)$ such that 
$$\dim E\leq 2n-s-t+(n-1)(n-2)/2=n(n-1)/2-(s+t-(n+1))$$ 
and for  $g\in O(n)\setminus E$,
\begin{equation}\label{eq4}
\mathcal L^n(\{z\in\Rn: \dim A\cap (g(B)+z)\geq s+t-n\})>0.
\end{equation}
\end{thm}

The version stated in the abstract concerning the case $\dim A + \dim B > n+1$ is slightly weaker than Theorem \ref{theo1}. Notice that the above upper bound for the dimension of $E$ is at least $(n-1)(n-2)/2$ which is the dimension of $O(n-1)$. In Section \ref{Examples} we show that it is needed in the estimates. The assumption $s+t>n+1$ only comes from the fact that the statement is trivial if $s+t\leq n+1$: then the above upper for $\dim E$ is at least $n(n-1)/2=\dim O(n)$ and we could take $E=O(n)$. 

This is an exceptional set estimate related to the following result of 
\cite{M2}: (\ref{eq4}) holds for $\theta_n$ almost all $g\in O(n)$ if one of the sets has dimension bigger than $(n+1)/2$, see also Chapter 13 in \cite{M4} and Chapter 7 in \cite{M5}. This of course is satisfied when $s+t>n+1$, as in the theorem.  It is expected that this generic result with respect to $\theta_n$ should hold whenever $\dim A + \dim B > n$. Under this condition it was proved (without exceptional set estimates) in \cite{K} and \cite{M1} provided the orthogonal group is replaced by a larger transformation group, for example by similarity maps as in \cite{M1}, or, more generally by Kahane in \cite{K}, by any closed subgroup of the general linear group acting transitively outside the origin. For the orthogonal group no dimensional restrictions are needed provided one of the sets satisfies some extra condition, for example  if it is rectifiable, see \cite{M1}, or a Salem set, see \cite{M3}.

It is easy to see, cf. the remark at the end of \cite{M2}, that in Theorem \ref{theo1} the positivity of the Hausdorff measures cannot be relaxed to $\dim A = s$ and $\dim B = t$.

If one of the sets supports a measure with sufficiently fast average decay over spheres for the Fourier transform, we can improve the estimate of Theorem \ref{theo1}. Then the results even hold for the sum sets provided the dimensions are big enough. This is given in Theorem \ref{theo3}. It yields immediately the following result in case one of the sets is a Salem set. By definition, $A$ is a Salem set if for every $0<s<\dim A$ there is $\mu\in\mathcal M(A)$ such that $|\mu(x)|^2\lesssim |x|^{-s}$. A discussion on Salem sets can be found, for example, in \cite{M5}, Section 3.6. 

\begin{thm}\label{theo4}
Let $A$ and $B$ be Borel subsets of $\Rn$  and suppose that $A$ is a Salem set. Suppose that $0<u<\dim A + \dim B - n$.

(a) If $\dim A+\dim B>2n-1$, then 
\begin{equation}\label{eq5}
\mathcal L^n(\{z\in\Rn: \dim A\cap (B+z)\geq u\})>0.
\end{equation}
(b)If $\dim A+\dim B\leq 2n-1$, then there is $E\subset O(n)$ with 
$$\dim E\leq n(n-1)/2-u$$ 
such that for  $g\in O(n)\setminus E$,
\begin{equation}\label{eq7}
\mathcal L^n(\{z\in\Rn: \dim A\cap (g(B)+z)\geq u\})>0.
\end{equation}
\end{thm}

Another consequence of Theorem \ref{theo3} is the following improvement of Theorem \ref{theo1} in the case where one of the sets has small dimension:

\begin{thm}\label{theo6}
Let $A$ and $B$ be Borel subsets of $\Rn$  and suppose that $\dim A\leq (n-1)/2$. If $0<u<\dim A + \dim B - n$, then there is $E\subset O(n)$ with 
$$\dim E\leq n(n-1)/2-u$$ 
such that for  $g\in O(n)\setminus E$,
\begin{equation}\label{eq18}
\mathcal L^n(\{z\in\Rn: \dim A\cap (g(B)+z)\geq u\})>0.
\end{equation}
\end{thm}

The method used to prove Theorem \ref{theo1} can easily be modified to other subgroups of the general linear group $GL(n)$ in place of the orthogonal group. For example, let $S(n)$ be the group of similarities, the compositions of orthogonal maps and dilations. Then  $\dim S(n)=n(n-1)/2+1$ and for any $x,z\in\Rn\setminus\{0\}$, the dimension of $\{g\in S(n): g(z)=x\}$ is the same as the dimension of $O(n-1)$, that is, $(n-1)(n-2)/2$. With small changes in the proof of Theorem \ref{theo1} this leads to

\begin{thm}\label{theo2}
Let $s$ and $t$ be numbers with $0<s,t<n$ and $s+t > n$. Let $A$ and $B$ be Borel subsets of $\Rn$ such that $\mathcal H^s(A)>0$ and $\mathcal H^t(B)>0$. Then 
there is $E\subset S(n)$ with 
$$\dim E\leq 2n-s-t+(n-1)(n-2)/2$$ 
and for  $g\in S(n)\setminus E$,
\begin{equation}\label{eq6}
\mathcal L^n(\{z\in\Rn: \dim A\cap (g(B)+z)\geq s+t-n\})>0.
\end{equation}
\end{thm}

\section{Peliminaries}

The proof of Theorem \ref{theo1} is based on the relationship of the Hausdorff dimension to the energies of measures and their relations to the Fourier transform. For $A\subset\Rn$ (or $A\subset O(n)$) we denote by $\mathcal M(A)$ the set of non-zero Radon measures $\mu$ on $\Rn$ with compact support $\spt\mu\subset A$. The Fourier transform of $\mu$ is defined by
$$\widehat{\mu}(x)=\int e^{-2\pi ix\cdot y}\,d\mu y,~ x\in\Rn.$$
For $0<s<n$ the $s$-energy of $\mu$ is 
\begin{equation}\label{eq10}
I_s(\mu)=\iint|x-y|^{-s}\,d\mu x\,d\mu y=c(n,s)\int|\widehat{\mu}(x)|^2|x|^{s-n}\,dx.
\end{equation} 
The second equality is a consequence of Parseval's formula and the fact that the distributional Fourier transform of the Riesz kernel $k_s, k_s(x)=|x|^{-s}$, is a constant multiple of $k_{n-s}$, see, for example, \cite{M4}, Lemma 12.12, or \cite{M5}, Theorem 3.10. These books contain most of the back-ground material needed in this paper.

We then have for any Borel set $A\subset\Rn$, cf. Theorem 8.9 in \cite{M4},
\begin{equation}\label{eq3}
\dim A=\sup\{s:\exists \mu\in\mathcal M(A)\ \text{such that}\ I_s(\mu)<\infty\}.
\end{equation}


Let $\nu\in\mathcal M(\Rn)$ and let $\psi_{\e}$ be an approximate identity: $\psi_{\e}(x)=\e^{-n}\psi(x/\e)$ where $\psi$ is a non-negative $C^{\infty}$-function with support in the unit ball and with integral $1$. Let $\nu_{\e}=\psi_{\e}\ast\nu$. Then the $\nu_{\e}$ converge weakly to $\nu$ when $\e\to 0$. Notice that $\widehat{\nu_{\e}}(x)=\widehat{\psi}(\e x)\widehat{\nu}(x)$. 

By the notation $M\lesssim N$ we mean that $M\leq CN$ for some constant $C$. The dependence of $C$ should be clear from the context. By $C(a)$ and $c(a)$ we mean positive constants depending on the parameter $a$. The closed ball with centre $x$ and radius $r$ will be denoted by $B(x,r)$. 
 
\begin{lm}\label{lemma2} Let $\theta\in\mathcal M(O(n))$ and $\a>(n-1)(n-2)/2$. If $\theta(B(g,r))\leq r^{\a}$ for all $g\in O(n)$ and $r>0$, then for $x,z\in\Rn\setminus\{0\}, r>0$,
\begin{equation}\label{eq11}
\theta(\{g:|x-g(z)|< r\})\lesssim (r/|z|)^{\a-(n-1)(n-2)/2}.
\end{equation}
\end{lm}
\begin{proof}
First we may clearly assume that $|z|=1$, and then also that $|x|=1$, because $|x-g(z)|< r$ implies $|x/|x|-g(z)|< 2r$. Then $O_{x,z}:=\{g\in O(n): g(z)=x\}$ can be identified with $O(n-1)$. Hence it is a smooth compact $(n-1)(n-2)/2$-dimensional submanifold of $O(n)$ which implies that it can be covered with roughly $r^{-(n-1)(n-2)/2}$ balls of radius $r$. If $g\in G$ satisfies $|x-g(z)|< r$, then $g$ belongs to the $r$-neighbourhood of $O_{x,z}$. The lemma follows from this.
\end{proof}

\section{Proof of Theorem \ref{theo1}}

The key to the proof of Theorem \ref{theo1} is the following energy estimate. 
For $\mu, \nu\in\mathcal M(\Rn), g\in O(n)$ and $z\in\Rn$, let $\nu_{\e}=\psi_{\e}\ast\nu$ as above and set

\begin{equation}\label{eq0}
\nu_{g,z,\e}(x)=\nu_{\e}(g^{-1}(x)-z),\ x\in\Rn.
\end{equation}

\begin{lm}\label{lemma1} Let $\beta>0$ and $\theta\in\mathcal M(O(n))$ be such that $\theta(O(n))\leq 1$ and for $x,z\in\Rn\setminus\{0\}, r>0,$ 
\begin{equation}\label{eq9}
\theta(\{g\in O(n):|x-g(z)|<r\})\leq (r/|z|)^{\beta}.
\end{equation}
Let $0<s,t<n, 0<u=s+t-n$ and $u>n-\beta$. Let $\mu, \nu \in \mathcal M(\Rn)$. Then
\begin{equation}\label{eq15}
\iint I_u(\nu_{g,z,\e}\mu)\,d\mathcal L^nz\,d\theta g \leq C(n,s,t)I_s(\mu)I_t(\nu).
\end{equation}

\end{lm}

\begin{proof}
We may assume that $I_s(\mu)$ and $I_t(\nu)$ are finite. 
Define 
$$\tilde{\nu}_{g,x,\e}(z)=\nu_{\e}(g^{-1}(x)-z),\ z\in\Rn.$$
Then 
$$\widehat{\tilde{\nu}_{g,x,\e}}(z)=e^{-2\pi ig^{-1}(x)\cdot z}\widehat{\nu_{\e}}(-z).$$
Hence by Parseval's formula for $x,y\in\Rn, x\neq y,$
\begin{align*}
&\int\nu_{\e}(g^{-1}(x)-z)\nu_{\e}(g^{-1}(y)-z)\,d\mathcal L^nz\\
&=\int\widehat{\tilde{\nu}_{g,x,\e}}(z)\overline{\widehat{\tilde{\nu}_{g,y,\e}}(z)}\,d\mathcal L^nz\\
&=\int\widehat{\nu_{\e}}(-z)\overline{\widehat{\nu_{\e}}(-z)}e^{-2\pi ig^{-1}(x-y)\cdot z}d\mathcal L^nz\\
\end{align*}
It follows by Fubini's theorem	 that
\begin{align*}
&I:=\iint I_u(\nu_{g,z,\e}\mu)\,d\mathcal L^nz\,d\theta g\\
&=\iiiint k_u(x-y)\nu_{\e}(g^{-1}(x)-z)\nu_{\e}(g^{-1}(y)-z)\,d\mu x\,d\mu y\,d\mathcal L^nz\,d\theta g\\
&=\iiint k_u(x-y)\left(\int\nu_{\e}(g^{-1}(x)-z)\nu_{\e}(g^{-1}(y)-z)\,d\mathcal L^nz\right)\,d\mu x\,d\mu y\,d\theta g\\
&=\iiint k_u(x-y)\left(\int|\widehat{\nu_{\e}}(z)|^2 e^{2\pi ig^{-1}(x-y)\cdot z}\,d\mathcal L^nz\right)\,d\mu x\,d\mu y\,d\theta g\\
&=\iiint k_{u,g,z}\ast\mu(x)\,d\mu x|\widehat{\nu_{\e}}(z)|^2\,d\mathcal L^nz\,d\theta g,
\end{align*}
where
\begin{equation*}
k_{u,g,z}(x)=|x|^{-u}e^{2\pi ig^{-1}(x)\cdot z}=|x|^{-u}e^{2\pi ix\cdot g(z)}
\end{equation*}
One checks by direct computation that the Fourier transform of $k_{u,g,z}$, in the sense of distributions, is given by 
\begin{equation*}
\widehat{k_{u,g,z}}(x)=
c(n,u)|x-g(z)|^{u-n}.
\end{equation*}
It follows that
$$\iint k_{u,g,z}\ast\mu\,d\mu=\int\widehat{k_{u,g,z}}|\widehat{\mu}|^2\,d\mathcal L^n=c(n,u)\int|x-g(z)|^{u-n}|\widehat{\mu}(x)|^2\,d\mathcal L^nx.$$
As $I_u(\mu)<\infty$, this is easily checked approximating $\mu$ with $\psi_{\e}\ast\mu$ and using the Lebesgue dominated convergence theorem. 
Thus 
\begin{equation}\label{eq8}
I=c(n,u)\iiint |x-g(z)|^{u-n}\,d\theta g|\widehat{\mu}(x)|^2|\widehat{\nu_{\e}}(z)|^2\,d\mathcal L^nx\,d\mathcal L^nz.
\end{equation}
We first observe that if $|x|\geq 2|z|$, then
$$\int |x-g(z)|^{u-n}\,d\theta g\leq \theta(O(n))2^{n-u}|x|^{u-n}\leq 2^{2n}|x|^{s-n}|z|^{t-n}.$$
Similarly if $|z|\geq 2|x|$, then
$$\int |x-g(z)|^{u-n}\,d\theta g\leq \theta(O(n))2^{n-u}|z|^{u-n}\leq 2^{2n}|x|^{s-n}|z|^{t-n}.$$
Suppose then that $|z|/2\leq|x|\leq2|z|.$ Then by the assumption
\begin{align*}
&\int|x-g(z)|^{u-n}\,d\theta g = \int_0^{\infty}\theta(\{g:|x-g(z)|^{u-n}>\lambda\})\,d\lambda\\
&=(n-u)\int_0^{\infty}\theta(\{g:|x-g(z)|<r\})r^{u-n-1}\,dr\\
&\lesssim \int_{0}^{|z|}(r/|z|)^{\beta}r^{u-n-1}\,dr+\int_{|z|}^{\infty}r^{u-n-1}\,dr\\
&\approx |z|^{u-n}\approx |x|^{s-n}|z|^{t-n},
\end{align*}
since $\beta+u-n>0$. 
It follows that  
\begin{equation}\label{eq13}
I\lesssim \iint |x|^{s-n}|z|^{t-n}|\widehat{\mu}(x)|^2|\widehat{\nu_{\e}}(y)|^2\,dx\,dy\lesssim I_s(\mu)I_t(\nu),
\end{equation}
 as required.

\end{proof}

Next we show that, with $\theta$ as above, for $\theta\times\mathcal L^n$ almost all $(g,z)$ the measures $\nu_{g,z,\e}\mu$ converge weakly as $\e\to 0$. It is immediate that for almost all $(g,z)$ this takes place through some sequence $(\e_j)$, depending on $(g,z)$, but we would at least need one sequence which is good for almost all $(g,z)$. The proof of the following theorem was inspired by an argument of Kahane in \cite{K}.

\begin{thm}\label{theo5}
Let $s,t$ and $u$ be positive numbers with $u=s+t-n>0$ and let $\mu, \nu\in\mathcal M(\Rn)$ with $I_s(\mu)<\infty$ and $I_t(\nu)<\infty$. Let $\psi_{\e}$ be an approximate identity and $\nu_{\e}=\psi_{\e}\ast\nu$. For $g\in O(n)$ and $z\in\Rn$, let  $\nu_{g,z,\e}$ be as in (\ref{eq0}). Finally, let $\theta\in\mathcal M(O(n))$ be as in Lemma \ref{lemma1}. Then for $\theta\times\mathcal L^n$ almost all $(g,z)$, as $\e\to 0$, the measures 
$\nu_{g,g^{-1}(z),\e}\mu$ converge weakly to a measure $\lambda_{g,z}$ with the properties
\begin{itemize}
\item[(a)] $$\spt\lambda_{g,z}\subset\spt\mu\cap(g(\spt\nu)+z),$$
\item[(b)]$$\int\lambda_{g,z}(\Rn)\,d\mathcal L^nz = \mu(\Rn)\nu(\Rn)\ \text{for}\ \theta\ \text{almost all}\ g\in O(n),$$
\item[(c)]
$$\iint I_u(\lambda_{g,z})\,d\mathcal L^nz\,d\theta g\leq C(n,s,t)I_s(\mu)I_t(\nu).$$ 
\end{itemize} 
\end{thm}
\begin{proof}
If the convergence takes place, the support property (a) is clear. Using the change of variable from $z$ to $g^{-1}(z)$ in the appropriate places, it is then sufficient to show that for $\theta\times\mathcal L^n$ almost all $(g,z)$, as $\e\to 0$, the measures 
$\nu_{g,z,\e}\mu$ converge weakly to a measure $\tilde{\lambda}_{g,z}$ such that (b) and (c) hold with $\lambda_{g,z}$ replaced by $\tilde{\lambda}_{g,z}$. 

Let $\phi\in C^+_0(\Rn)$. Then by Lemma \ref{lemma1},
$$\iint (\int\nu_{g,z,\e}\phi\,d\mu)^2\,d\mathcal L^nz\,d\theta g\lesssim I_s(\mu)I_t(\nu)<\infty.$$ 
Hence by Fatou's lemma
$$\int \left(\liminf_{\e\to 0}\int(\int\nu_{g,z,\e}\phi\,d\mu)^2\,d\mathcal L^nz\right)\,d\theta g\lesssim I_s(\mu)I_t(\nu)<\infty.$$
Thus for $\theta$ almost all $g$ there is a sequence $(\e_j)$ tending to $0$ such that 
$$\sup_j\int(\int\nu_{g,z,\e_j}\phi\,d\mu)^2\,d\mathcal L^nz<\infty.$$
On the other hand, defining the measure $\mu_{\phi,g}$ by $\int h\,d\mu_{\phi,g}=\int h(-g^{-1}(x))\phi(x)\,d\mu x$, we have
$$\int\nu_{g,z,\e}\phi\,d\mu=\int\nu_{\e}(g^{-1}(x)-z)\phi(x)\,d\mu x=\mu_{\phi,g}\ast\nu\ast\psi_{\e}(-z),$$
and the measures $\mu_{\phi,g}\ast\nu\ast\psi_{\e}$ converge weakly to $\mu_{\phi,g}\ast\nu$. Consequently, $\mu_{\phi,g}\ast\nu$ is an $L^2$ function on $\Rn$ and the convergence takes place almost everywhere. It follows now that for $\theta$ almost all $g\in O(n)$ and for every $\phi\in C^+_0(\Rn)$ the finite limit
\begin{equation}\label{eq2}
L_{g,z}\phi:=\lim_{\e\to 0}\int\nu_{g,z,\e}\phi\,d\mu=\lim_{\e\to 0}\mu_{\phi,g}\ast\nu\ast\psi_{\e}(-z)=\mu_{\phi,g}\ast\nu(-z)
\end{equation}
exists for almost all $z\in\Rn$. Let $\mathcal D$ be a countable dense subset of $C^+_0(\Rn)$ containing a function $\phi_0$ which is $1$ on the support of $\mu$. Then there is a set $E$ of measure zero such that (\ref{eq2}) holds for all $z\in\Rn\setminus E$ for all $\phi\in \mathcal D$, the exceptional set is indenpendent of $\phi$. Applying (\ref{eq2}) to $\phi_0$ we see that
$$\sup_{\e>0}\int\nu_{g,z,\e}\,d\mu<\infty.$$
Then by the Cauchy criterion the denseness of $\mathcal D$ yields that whenever $z\in\Rn\setminus E$, there is the finite limit $L_{g,z}\phi:=\lim_{\e\to 0}\int\nu_{g,z,\e}\phi\,d\mu$ for all $\phi\in C^+_0(\Rn)$. Hence by the Riesz representation theorem the positive linear functional $L_{g,z}$ corresponds to a Radon measure $\tilde{\lambda}_{g,z}$ to which the measures $\nu_{g,z,\e}\mu$ converge weakly.

The claim (b) follows from (\ref{eq2}):
\begin{align*}
\int\lambda_{g,z}(\Rn)\,d\mathcal L^nz = &\int L_{g,z}\phi_0\,d\mathcal L^nz 
= \int\mu_{\phi_0,g}\ast\nu(-z)\,d\mathcal L^n\\
&=\mu_{\phi_0,g}(\Rn)\nu(\Rn)=\mu(\Rn)\nu(\Rn).	
\end{align*}
The claim (c) follows from Lemma \ref{lemma1}, Fatou's lemma and the lower semicontinuity of the energy-integrals under the weak convergence.
\end{proof}

\begin{proof}[Proof of Theorem \ref{theo1}]
Theorem \ref{theo1} follows from Lemma \ref{lemma1} and Theorem \ref{theo5}: Let 
$$G=\{g\in O(n): \mathcal L^n(\{z\in\Rn: \dim A\cap (g(B)+z))\geq s+t-n\})=0\}.$$
Then $G$ is a Borel set. We leave checking this to the reader. It is a bit easier when $A$ and $B$ are compact. We may assume the compactness since $A$ and $B$ as in the theorem contain compact subsets with positive measure, cf \cite{Fe}, 2.10.48. Suppose, contrary to what is claimed, that $\dim G > 2n-s-t+(n-1)(n-2)/2.$ Let 
$\dim G > \a > 2n-s-t + (n-1)(n-2)/2$. Then by Frostman's lemma, cf. \cite{M4}, Theorem 8.8, and Lemma \ref{lemma2} there is $\theta\in\mathcal M(G)$ satisfying (\ref{eq9}) with $\beta=\a-(n-1)(n-2)/2>2n-s-t$. 
By Frostman's lemma there are $\mu\in\mathcal M(A)$ and $\nu\in\mathcal M(B)$ such that $\mu(B(x,r))\leq r^s$ and $\nu(B(x,r))\leq r^t$ for all balls $B(x,r)$. Then by easy estimation, for example, as in the beginning of Chapter 8 in \cite{M4}, $I_{s'}(\mu)<\infty$ and $I_{t'}(\nu)<\infty$ for $0<s'<s$ and $0<t'<t$. By Theorem \ref{theo5}(b) the set $E_g=\{z:\lambda_{g,z}(\Rn)>0\}$ has positive Lebesgue measure for $\theta$ almost all $g$. It then follows from Theorem \ref{theo5}(a) and (c) and (\ref{eq3}) that for $\theta$ almost all  $g$,~ $\dim A\cap(g(B)+z)\geq s+t-n$ for almost all $z\in E_g$. This contradicts the definition of $G$ and the fact that $\theta$ has support in $G$. 
\end{proof}

%

\section{Intersections and the decay of spherical averages}\label{decay}

For $\mu\in\mathcal M(\Rn)$ define the spherical averages
$$\sigma(\mu)(r)=r^{1-n}\int_{S(r)}|\widehat{\mu}(x)|^2\,d\sigma_r^{n-1}x, r>0,$$
where $\sigma_r^{n-1}$ is the surface measure on the sphere $S(r)=\{x\in\Rn:|x|=r\}$. 
Notice that if $\sigma(\mu)(r)\lesssim r^{-\gamma}$ for $r>0$ and for some $\gamma>0$, then $I_s(\mu)<\infty$ for $0<s<\gamma$. We now prove that under such decay condition we can improve Theorem \ref{theo1}:

\begin{thm}\label{theo3}
Let $t,\gamma$ and $\Gamma$ be positive numbers with $t,\gamma<n$. Let $\mu, \nu\in\mathcal M(\Rn)$ with $\sigma(\mu)(r)\leq \Gamma r^{-\gamma}$ for $r>0, I_{\gamma}(\mu)<\infty$ and $I_t(\nu)<\infty$. 

(a) If $\gamma+t>2n-1$, then 
\begin{equation}\label{eq16}
\mathcal L^n(\{z\in\Rn: \dim \spt\mu\cap (\spt\nu+z)\geq \gamma+t-n\})>0.
\end{equation}
(b) If $\gamma+t\leq 2n-1$, then there is $E\subset O(n)$ with 
$$\dim E\leq 2n-1-\gamma-t+(n-1)(n-2)/2$$ 
such that for  $g\in O(n)\setminus E$,
\begin{equation}\label{eq17}
\mathcal L^n(\{z\in\Rn: \dim \spt\mu\cap (g(\spt\nu)+z))\geq \gamma+t-n\})>0.
\end{equation}
\end{thm}

\begin{proof} Let $u=\gamma+t-n$. 
As above, we only need to show that the the conclusion of Lemma \ref{lemma1} holds under the present assumptions, but now the upper bound in (\ref{eq15}) will be a constant involving $\Gamma, I_t(\nu),I_{\gamma}(\mu), I_u(\mu),\mu(\Rn)$ and $\nu(\Rn)$. 
For the statement (a) there is no $\theta$ integration (or $\theta$ is the Dirac measure at the identity map) and $n-1<u<n$, and for the statement (b) $\theta$ satisfies (\ref{eq9}) with some $\beta$ with $n-1-u<\beta<n-u$. 

We again have (\ref{eq8}). The integration over $|z|\leq 2$ is easily controlled for example by
\begin{align*}
&\iiint_{|z|\leq 2} |x-g(z)|^{u-n}\,d\theta g|\widehat{\mu}(x)|^2|\widehat{\nu_{\e}}(z)|^2\,d\mathcal L^nx\,d\mathcal L^nz\\
&\lesssim \mu(\Rn)^2\nu(\Rn)^2\iint_{|z|\leq 2, |x|\leq 3} |x-g(z)|^{u-n}\,d\mathcal L^nx\,d\mathcal L^nz+
\nu(\Rn)^2\int |x|^{u-n}|\widehat{\mu}(x)|^2\,d\mathcal L^nx\\
&\lesssim (\mu(\Rn)^2+I_u(\mu))\nu(\Rn)^2.
\end{align*}
For the part where $|x|>2|z|$ or $|z|>2|x|$ we can argue as before. 

To prove (a), suppose $n-1<u<n$. Then it suffices to show 
$$\iint_{|z|/2\leq |x|\leq 2|z|, |z|>2} |x-z|^{u-n}|\widehat{\mu}(x)|^2|\widehat{\nu_{\e}}(z)|^2\,d\mathcal L^nx\,d\mathcal L^nz\lesssim \Gamma I_t(\nu).$$
Since $-1<u-n<0, n-1-\gamma\leq t-n$ and $|x-z|\geq||x|-|z||$, the integral over the part $||x|-|z||\leq 1, |z|>1,$ can be estimated by
\begin{align*}
&\iint_{||x|-|z||\leq 1, |z|>2} |x-z|^{u-n}|\widehat{\mu}(x)|^2|\widehat{\nu_{\e}}(z)|^2\,d\mathcal L^nx\,d\mathcal L^n\,z\\
&\leq \int_{|z|>2}\int_{|z|-1}^{|z|+1}|r-|z||^{u-n}\int_{S(r)}|\widehat{\mu}(x)|^2\,d\sigma_r^{n-1}x\,dr|\widehat{\nu_{\e}}(z)|^2\,d\mathcal L^nz\\
&\lesssim \Gamma\int_{|z|>2} |z|^{n-1-\gamma}|\widehat{\nu_{\e}}(z)|^2\,d\mathcal L^nz\\
&\leq \Gamma\int |z|^{t-n}|\widehat{\nu_{\e}}(z)|^2\,d\mathcal L^nz\lesssim \Gamma I_t(\nu).
\end{align*}
For the remaining part we have 
\begin{align*}
&\iint_{||x|-|z|| > 1, 1<|z|/2\leq|x|\leq 2|z|} |x-z|^{u-n}|\widehat{\mu}(x)|^2|\widehat{\nu_{\e}}(z)|^2\,d\mathcal L^nx\,d\mathcal L^nz\\
&\leq\int\sum_{1\leq2^j\leq3|z|} \int_{2^j\leq||x|-|z||\leq 2^{j+1},|z|/2\leq|x|\leq 2|z|}||x|-|z||^{u-n}|\widehat{\mu}(x)|^2|\widehat{\nu_{\e}}(z)|^2\,d\mathcal L^nx\,d\mathcal L^nz\\
&\lesssim \int\sum_{1\leq2^j\leq3|z|} 2^{j(u-n)}\int_{2^j\leq||x|-|z||\leq 2^{j+1},|z|/2\leq|x|\leq 2|z|}|\widehat{\mu}(x)|^2\,d\mathcal L^nx|\widehat{\nu_{\e}}(z)|^2\,d\mathcal L^nz\\
&= \int\sum_{1\leq2^j\leq3|z|} 2^{j(u-n)}\int_{2^j\leq|r-|z||\leq 2^{j+1},|z|/2\leq r\leq 2|z|}\int_{S(r)}|\widehat{\mu}(x)|^2\,d\sigma_r^{n-1}x\,dr|\widehat{\nu_{\e}}(z)|^2\,d\mathcal L^nz\\
&\lesssim \Gamma \int\sum_{1\leq2^j\leq3|z|} 2^{j(u-n)}2^j|z|^{n-1-\gamma}|\widehat{\nu_{\e}}(z)|^2\,d\mathcal L^nz\\
&\lesssim \Gamma\int |z|^{u-\gamma}|\widehat{\nu_{\e}}(z)|^2\,d\mathcal L^nz\lesssim \Gamma I_t(\nu).
\end{align*}

We establish the statement (b) by showing that 
$$\iiint_{|z|/2\leq |x|\leq 2|z|,|z|>2} |x-g(z)|^{u-n}\,d\theta g|\widehat{\mu}(x)|^2|\widehat{\nu_{\e}}(z)|^2\,d\mathcal L^nx\,d\mathcal L^nz\lesssim \Gamma I_t(\nu),$$
where $\theta\in\mathcal M(O(n))$ is as in Lemma \ref{lemma1} with $n-1-u<\beta<n-u$. 
We first have as in the proof of Lemma \ref{lemma1}
\begin{align*}
&\int|x-g(z)|^{u-n}\,d\theta g =(n-u)\int_0^{\infty}\theta(\{g:|x-g(z)|<r\})r^{u-n-1}\,dr\\
&\lesssim \int_{||x|-|z||}^{\infty}(r/|z|)^{\beta}r^{u-n-1}\,dr=(n-\beta-u)||x|-|z||^{\beta+u-n}|z|^{-\beta},
\end{align*} 
because $\theta(\{g:|x-g(z)|<r\})=0$ if $r<||x|-|z||$. 
Using $-\beta+n-1-\gamma<t-n$  the integral over the part $||x|-|z||\leq 1$ can be estimated by
\begin{align*}
&\iiint_{||x|-|z||\leq 1,|z|>2} |x-g(z)|^{u-n}\,d\theta g|\widehat{\mu}(x)|^2|\widehat{\nu_{\e}}(x)|^2\,d\mathcal L^nx\,d\mathcal L^n\,z\\
&\lesssim\int_{|z|>2} \int_{|z|-1}^{|z|+1}|r-|z||^{\beta+u-n}|z|^{-\beta}\int_{S(r)}|\widehat{\mu}(x)|^2\,d\sigma_r^{n-1}x\,dr|\widehat{\nu_{\e}}(z)|^2\,d\mathcal L^nz\\
&\lesssim \Gamma\int_{|z|>2} |z|^{-\beta+n-1-\gamma}|\widehat{\nu_{\e}}(z)|^2\,d\mathcal L^nz\\
&\leq \Gamma\int |z|^{t-n}|\widehat{\nu_{\e}}(z)|^2\,d\mathcal L^nz
\lesssim \Gamma I_t(\nu).
\end{align*}
For the remaining part we have 

\begin{align*}
&\iiint_{||x|-|z|| > 1, 1<|z|/2\leq|x|\leq 2|z|} |x-g(z)|^{u-n}\,d\theta g|\widehat{\mu}(x)|^2|\widehat{\nu_{\e}}(z)|^2\,d\mathcal L^nx\,d\mathcal L^nz\\
&\leq\int\sum_{1\leq2^j\leq3|z|} \int_{2^j\leq||x|-|z||\leq 2^{j+1},|z|/2\leq|x|\leq 2|z|}||x|-|z||^{\beta+u-n}|z|^{-\beta}|\widehat{\mu}(x)|^2|\widehat{\nu_{\e}}(z)|^2\,d\mathcal L^nx\,d\mathcal L^nz\\
&\lesssim \int\sum_{1\leq2^j\leq3|z|} 2^{j(\beta+u-n)}|z|^{-\beta}\int_{2^j\leq||x|-|z||\leq 2^{j+1},|z|/2\leq|x|\leq 2|z|}|\widehat{\mu}(x)|^2\,d\mathcal L^nx|\widehat{\nu_{\e}}(z)|^2\,d\mathcal L^nz\\
&= \int\sum_{1\leq2^j\leq3|z|} 2^{j(\beta+u-n)}|z|^{-\beta}\int_{2^j\leq|r-|z||\leq 2^{j+1},|z|/2\leq r\leq 2|z|}\int_{S(r)}|\widehat{\mu}(x)|^2\,d\sigma_r^{n-1}x\,dr|\widehat{\nu_{\e}}(z)|^2\,d\mathcal L^nz\\
&\lesssim \Gamma\int\sum_{1\leq2^j\leq3|z|} 2^{j(\beta+u-n)}|z|^{-\beta}2^j|z|^{n-1-\gamma}|\widehat{\nu_{\e}}(z)|^2\,d\mathcal L^nz\\
&\lesssim \Gamma\int |z|^{u-\gamma}|\widehat{\nu_{\e}}(z)|^2\,d\mathcal L^nz\lesssim \Gamma I_t(\nu).
\end{align*}
\end{proof}

For $0<s<n$ denote by $\gamma_n(s)$ the supremum of the numbers $\gamma$ such that
\begin{equation}\label{eq19}
\sigma(\mu)(r)\lesssim I_s(\mu)r^{-\gamma}\ \text{for}\ r>0
\end{equation}
holds for all $\mu\in\mathcal M(\Rn)$ with support in the unit ball. Estimates for $\gamma(s)$ are discussed in \cite{M5}, Chapter 15. For $s\leq (n-1)/2$ the optimal, rather easy, result $\gamma(s)=s$ is valid, see \cite{M5}, Lemma 3.15. This together with Theorem \ref{theo3} yields immediately Theorem \ref{theo6}. For $s>(n-1)/2$ the optimal estimate fails and Theorem \ref{theo3} only gives a lower bound for the dimension of intersections which stays below and bounded away from $\dim A + \dim B - n$. The deepest estimates are due to Wolff  \cite{W} in the plane and to Erdo\u gan \cite{E} in higher dimensions. They give $\gamma_n(s)\geq (n+2s-2)4$ for $n/2\leq s \leq (n+2)/2$. Theorem \ref{theo3} combined with this leads to the result that if $\dim A + \dim B/2 - (3n+2)/4>u>0$, then  
$$\mathcal L^n(\{z\in\Rn: \dim A\cap (g(B)+z)) > u\})>0$$
for $g\in O(n)$ outside an exceptional set $E$ with $\dim E \leq n(n-1)/2 - u$. Plugging into Theorem \ref{theo3} other known estimates for $\gamma_n(s)$ gives similar rather weak intersection results.


\section{Examples}\label{Examples}

The first example here shows that in Theorem \ref{theo1} the bound $(n-1)(n-2)/2$ is sharp in the case where both sets have the maximal dimension $n$. This of course does not tell us anything in the plane but it explains the appearance of the dimension of $O(n-1)$. 
In the following we identify $O(n-1)$ with a subset of $O(n)$ letting $g\in O(n-1)$ mean the map $(x_1,\dots,x_n)\mapsto (g(x_1\dots,x_{n-1}),x_n)$. 
\begin{ex}\label{ex1}
Let $n\geq 3$. There are compact sets $A, B\subset\R^n$ such that $\dim A=\dim B=n$ and for every $g\in O(n-1)$,~  $\dim A\cap (g(B)+z)\leq n-1$ for  all $z\in\R^n$.
\end{ex}
\begin{proof}
Let $C, D\subset\R$ be compact sets such that $\dim C = \dim D = 1$ and for every $z\in\R$ the intersection $C\cap(D+z)$ contains at most one point. Such sets were constructed in \cite{M1}, the construction is explained also in  \cite{M4}, Example 13.18. Let $F$ be the closed unit ball in $\R^{n-1}$ and take $A=F\times C$ and $B=F\times D$. These sets have the required properties.
\end{proof}

The following example shows that we need some additional assumptions, for example as in Theorem \ref{theo3}, to get any result using only translations:

\begin{ex}\label{ex2}
There are compact sets $A, B\subset\R^n$ such that $\dim A = \dim B =n$ and for every $z\in\R^n$ the intersection $A\cap(B+z)$ contains at most one point.
\end{ex}
\begin{proof}
Let $C, D\subset\R$ be the compact sets of the previous example. Take $A=C^n$ and $B=D^n$. These sets have the required properties.
\end{proof}

I do not know what are the sharp bounds for the dimension of exceptional sets of Theorem \ref{theo1}. For simplicity, let us look at this question in the plane. Let $d(s,t)\in[0,1], 0<s,t\leq 2, s+t>2,$ be the infimum of the numbers $d>0$ with the property that for all Borel sets $A,B\subset\R^2$ with $\dim A = s, \dim B = t$ and for all $0<u<s+t-2$, 
$$\dim\{g\in O(2): \mathcal L^2(\{z\in\R^2:\dim A\cap(g(B)+z)>u\})=0\}\leq d.$$
The problem is to determine $d(s,t)$. We know from Theorem \ref{theo1} that if $s+t>3$, then $d(s,t)\leq 4-s-t$. In particular  $d(2,2)=0$. This suggests that $d(s,t)$ might be $4-s-t$ when $s+t>3$. However we know from Theorem \ref{theo6} that $d(s,t)\leq 3-s-t$  whenever $s\leq 1/2$.  I would be happy to see some examples sheding light into this question.

\section{Concluding remarks}
As mentioned in the Introduction, intersection problems of this type for general sets were first studied by Kahane in \cite{K} and by the author in \cite{M1}, involving transformations such as similarities. For the orthogonal group the result in \cite{M2} concerned the case where one of the sets has dimension bigger than $(n+1)/2$. In \cite{M3} a general method was developed to get dimension estimates for the distance sets and intersections once suitable spherical average estimates	(\ref{eq19}) for measures with finite energy are available. Such deep estimates were proved by Wolff in \cite{W} and Erdo\u gan in \cite{E}. They gave the best known results for the distance sets, see \cite{M5}, Chapters 15 and 16, but only minor progress for the intersections, as mentioned in Section \ref{decay}.  The known estimates for $\sigma(\mu)$ are discussed in  \cite{M5}, Chapter 15, see also \cite{LR} for a recent one. 

The reverse inequality in Theorem \ref{theo1} fails: for any $0\leq s\leq n$ there exists a Borel set $A\subset \Rn$ such that 
$\dim A\cap f(A)=s$ for all similarity maps $f$ of $\Rn$. This follows from \cite{F2}, see also Example 13.19 in \cite{M4} and the further references given there. The reverse inequality holds if one of the sets is a reasonably nice integral dimensional set, for example rectifiable, or if $\dim A\times B = \dim A + \dim B$, see \cite{M1}. This latter condition is valid if, for example, one of the sets is Ahlfors-David regular, see \cite{M4}, pp. 115-116. For such reverse inequalities no rotations $g$ are needed (or, equivalently, they hold for every $g$).

Exceptional set estimates in the spirit of this paper were first proved for projections by Kaufman in \cite{Ka}, then continued by Kaufman and the author \cite{KM} and by Falconer \cite{F1}. Peres and Schlag \cite{PS} proved such estimates for large classes of generalized projections. Exceptional set estimates for intersections with planes were first proved by Orponen \cite{O1} and continued by Orponen and  the author \cite{MO}. In \cite{O2} Orponen derived estimates for radial projections. All these estimates expect those in \cite{MO} and some in \cite{PS} are known to be sharp. Some of these and other related results are also discussed in \cite{M5}.

Recently Donoven and Falconer \cite{DF} investigated Hausdorff dimension of intersections for subsets of certain Cantor sets and Shmerkin and Suomala \cite{SS} for large classes of random sets.

\vspace{1cm}
\begin{footnotesize}
{\sc Department of Mathematics and Statistics,
P.O. Box 68,  FI-00014 University of Helsinki, Finland,}\\
\emph{E-mail address:} 
\verb"pertti.mattila@helsinki.fi" 
"

\end{footnotesize}

\end{document}